\documentclass{svmult}

\usepackage{graphicx}
\usepackage{amssymb}
\usepackage{amsmath}
\usepackage{epsfig}
\usepackage{enumerate}
\usepackage{float}
\usepackage{mathptmx}
\usepackage{bm}
\usepackage{multirow}
\makeindex             % used for the subject index
\begin{document}

\title*{Parameter Identification of Constrained Data by a New Class of Rational Fractal Function}
\titlerunning{Parameter Identification by Rational Fractal Function}
 %Use \titlerunning{Short Title} for an abbreviated version of
\author{S. K. Katiyar$^\dag$, A. K. B. Chand$^{\dag\bigstar}$ and Sangita Jha$^\dag$}
\institute{Submitted to ICMC-2017, Haldia\\\\$^\dag$Department of
Mathematics, Indian Institute of Technology Madras, Chennai -
600036, India\\\email{sbhkatiyar@gmail.com, chand@iitm.ac.in, sangitajha285@gmail.com}}
\maketitle
\abstract{This paper sets a
theoretical foundation for the applications of the fractal interpolation functions (FIFs). We construct rational cubic spline FIFs (RCSFIFs) with quadratic denominator involving two shape parameters. The elements of the iterated function system (IFS) in each subinterval are identified befittingly so that the graph of the resulting $\mathcal{C}^1$-RCSFIF  lies within a prescribed rectangle. These parameters include, in particular, conditions on the positivity of the $\mathcal{C}^1$-RCSFIF. The problem of visualization of constrained data is also addressed when the data is lying above a straight line, the proposed fractal curve is required to lie on the same side of the line. We illustrate our interpolation scheme with some numerical examples.}
\noindent\textbf{Keywords} Iterated Function System. Fractal Interpolation Functions. Rational cubic fractal functions. Rational cubic interpolation. Constrained Interpolation. Positivity
\section{Introduction}\label{HALDIASANGEETAsec1}
It is often desirable, or even strictly required for physical interpretation, that a interpolating function preserves certain salient shape properties of scientific data such as convexity, monotonicity, positivity and constrained by a line or a surface. This problem of searching a sufficiently smooth function that guarantees the preservation of geometric shape properties inherent in given data is generally referred to as shape preserving interpolation/approximation, which receives considerable attention in research area of mathematics, computer science and engineering design. Without changing prescribed data or inserting additional knots, rational splines (see, for instance, \cite{DQ,DQ1,GD85,FC,SH,SHN,SHH}, and references therein) have been successfully replaced the ordinary polynomials in shape-preserving surroundings due to their ability to modify the shape of curves and surfaces via tension parameters. Moreover, they provide a powerful techniques among various techniques available in the classical 
numerical analysis for designing of curves, surfaces and some analytic primitives such as conic sections that are widely used in engineering design, data visualization, cartography and various computer graphics applications. A significant work in this direction by many authors has been contributed and there is a plethora of papers on its versatility in the literature.
\par
\noindent The word fractal interpolation function  was introduced by Barnsley \cite{B1} in 1986 based on the theory of IFS. He initiated the way to define a FIF as the fixed point of the Read-Bajraktarevi\'{c} operator defined on a suitable space of functions which becomes day by day a very powerful tool for approximating irregular objects. FIFs, different from traditional interpolation functions in the sense that their graphs are typical fractals or self-referential sets. The novelty of fractal interpolants lies in their ability to model a data set with either a class of smooth or a nonsmooth functions depending on the problem at hand. Further, differentiable FIFs \cite{B2} initiated a striking relationship between fractal functions and traditional nonrecursive interpolants. Thereafter, many authors
have worked in the area of constructing various types of FIFs, including spline FIFs  (see, for instance, \cite{CV2,N1,N4,NS3,NVCSS,WY}) and hidden variable FIFs \cite{B3,BPL,CK6,CK7}). In this way, the fractal methodology provides more flexibility and versatility on the choice of interpolant. Consequently, this function class can be useful for mathematical and engineering problems where classical spline interpolation approach may not work satisfactorily. To broaden their horizons, some special class of fractal interpolants are introduced and their shape preserving aspects are investigated recently in the literature. As a submissive contribution to this goal, Chand and coworkers have initiated the study on shape preserving fractal interpolation and approximation using various families of polynomial and rational IFSs (see, for instance, \cite{CV2,CVN,CNVS}).
\par
\noindent Constraining an interpolation to be shape preserving is a well established technique for modeling scientific data. The existing research on shape preserving FIFs mainly concern about the three important shape properties, namely, positivity, monotonicity, and convexity. The fractal functions are not well explored in the field of constrained interpolation and the current article is an attempt in this direction. This paper presents a description and analysis of a RCSFIF that has two shape parameters associated with each interval. The elements of the IFS in each subinterval are identified befittingly so that the graph of the resulting $\mathcal{C}^1$-RCSFIF  lies within a prescribed rectangle. The problem of keeping the graph of $\mathcal{C}^1$-RCSFIF within a rectangle includes, in particular, the problem of positivity of the $\mathcal{C}^1$-RCSFIF and hence, improves the sufficient conditions for the positivity already appeared in \cite{SHH}. The problem of visualization of constrained data is also 
addressed when the data is lying above a straight line, the corresponding fractal curves are required to lie on the same side of the line. To obtain the visually desirable shape, scaling factors and shape parameters can be adjusted by using optimization techniques. The $\mathcal{C}^1$-RCSFIF scheme is, in general, global because the variations in the scaling factor of a particular subinterval may influence the entire configuration. However, it recovers the traditional rational interpolation scheme \cite{SHH}, when the scaling factor in each subinterval is taken to be zero.
\section{Basic facts}\label{HALDIASANGEETAsec2}
In this section we introduce the basic terminologies required for our work. We also state the intermediate propositions corresponding to the main
steps of our argument. For a more extensive treatment, the reader may consult \cite{B1,B2,PRM}.
\subsection{IFS for fractal functions}\label{HALDIASANGEETAsubsec2a}
For $r\in \mathbb{N}$, let  $\mathbb{N}_r$ denote the subset $\{1,2,\dots, r\}$ of $\mathbb{N}$. Let a set of data points $\mathcal{D}=\{(x_i, y_i) \in \mathbb{R}^2: i \in \mathbb{N}_N\}$ satisfying $ x_1<x_2<\dots<x_N$, $N>2$, be given. Set $I = [x_1,x_N]$, $I_i = [x_i,x_{i+1}]$ for $i \in \mathbb{N}_{N-1}$. Suppose $L_i: I \rightarrow I_i$, $i \in \mathbb{N}_{N-1}$ be contraction homeomorphisms such that
\begin{equation}\label{HALDIASANGEETAeq1}
L_i(x_1)=x_i,\ L_i(x_N)=x_{i+1}.
\end{equation}
For instance, if $L_i(x) = a_i x+ b_i$, then the prescriptions in
(\ref{HALDIASANGEETAeq1}) yield
\begin{equation}\label{HALDIASANGEETAeq2}
a_i=\frac{x_{i+1}-x_{i}}{x_N-x_1},~~b_i=\frac{x_N
x_i-x_1x_{i+1}}{x_N-x_1},~~ i\in \mathbb{N}_{N-1}.
\end{equation}
Let $0<r_i<1, i\in \mathbb{N}_{N-1}$, and $X:=I \times \mathbb{R}$. Let $N-1$ continuous mappings
 $F_i: X \to \mathbb{R}$ be given satisfying:
 \begin{equation}\label{HALDIASANGEETAeq3}
F_i(x_1,y_1)=y_i,\ \ F_i(x_N,y_N)=y_{i+1},~~ \vert
F_i(x,y)-F_i(x,y^*)\vert \leq r_i \vert
 y-y^*\vert,
\end{equation}
where $(x,y), (x,y^*)\in X$. Define $w_i: X \to I_i \times \mathbb{R}\subseteq X$,\ ~$w_i(x,y)=\big(L_i(x),F_i(x,y)\big)$
$\forall~i \in \mathbb{N}_{N-1} $. It is known \cite{B1} that there exists a metric on $\mathbb{R}^2$, equivalent to the Euclidean metric, with respect to which $w_i, i\in \mathbb{N}_{N-1}$, are contractions. The collection $\mathcal{I}=\{X; w_i,i \in \mathbb{N}_{N-1}\}$ is called an IFS. \noindent Associated with the IFS $\mathcal{I}$, there is a set
valued Hutchinson map $W:H(X) \rightarrow H(X)$ defined by $W(B)=\underset{i=1}  {\overset{N-1} \cup}w_i(B)$ for  $B\in H(X)$, where $H(X)$ is the set of all nonempty compact subsets of $X$ endowed with the Hausdorff metric $h_d$. The Hausdorff distance between A and B in $\textit{H(X)}$ is defined by
$h_d(A,B)=\max
\big\{\underset{a\in A} \max \;\underset{b \in B} \min \;d(a,b),
\underset{b\in B} \max\; \underset{a \in A} \min \;d(b,a)\big\}.$ The complete metric space $\textit({H(X)},h_d)$ is called the \textit{space of fractals}. Further, $W$ is a contraction map on the complete metric space $(H(X), h_d)$. By the Banach Fixed Point Theorem, there exists a unique set $G\in H(X)$ such that $W(G) = G$. This set $G$ is called the
attractor or deterministic fractal corresponding to the IFS $\mathcal{I}$. For any choices of $L_i$ and $F_i$
satisfying the conditions prescribed in (\ref{HALDIASANGEETAeq1})-(\ref{HALDIASANGEETAeq3}) , the following result holds.
\begin{proposition}\label{HALDIASANGEETAprop1}(Barnsley \cite{B1})
The IFS $\{X; w_i; i\in \mathbb{N}_{N-1}\}$ defined above admits a unique
attractor $G,$ and $G$ is the graph of a continuous function
$g:I\to \mathbb{R}$ which obeys $g(x_i)=y_i$\;for $i \in \mathbb{N}_N$.
\end{proposition}
\begin{definition}
The aforementioned function $g$  whose graph is the attractor of an IFS is called a \textbf{fractal interpolation function} (FIF) or a \textbf{self-referential function} corresponding to the IFS $\{X;\ w_i;i\in \mathbb{N}_{N-1}\}.$
\end{definition}
The above FIF $g$ is obtained as the fixed point of the Read-Bajraktarevi\'{c} (RB) operator $T$ defined on a complete metric space $(\mathcal{G}, \rho)$:
\begin{equation*}
(Th^*) (x)=  F_i\left(L_i^{-1}(x), h^*\circ L_i^{-1}(x)\right)\;
\forall~ x\in I_i ,\; i\in \mathbb{N}_{N-1}.
\end{equation*}
where $\mathcal{G} :=\{h^*: I \rightarrow \mathbb{R}: h^*~ \text{is continuous on}~ I,~ h^*(x_1)=y_1, h^*(x_N)=y_N\}$ is equipped with the uniform metric.
It can be seen that $T$ is a contraction mapping on $(\mathcal{G}, \rho)$ with a contraction factor $r^*:= \max\{r_i: i \in \mathbb{N}_{N-1}\}<1$. The fixed point of $T$ is the FIF $g$ corresponding to the IFS $\mathcal{I}$. Therefore, $g$ satisfies the functional equation:
\begin{equation}\label{HALDIASANGEETAeq4}
g(x) = F_i\left(L_i^{-1}(x),g \circ L_i^{-1}(x)\right),\; x\in
I_i,\; i\in \mathbb{N}_{N-1},
\end{equation}
The most extensively studied FIFs so far in the literature stem from the IFS
\begin{equation}\label{HALDIASANGEETAeq5}
\mathcal{I}=\{X; w_i(x,y)\equiv (L_i(x)=a_ix+b_i,~  F_i(x,y)=\alpha_i y + q_i(x));i \in \mathbb{N}_{N-1}\}.
\end{equation}
Here $q_i : I \to \mathbb{R}$ are suitable continuous functions, generally polynomials, satisfying (\ref{HALDIASANGEETAeq3}). If $q_i$ are polynomials, then the IFS is referred to as a polynomial IFS and the corresponding FIF is termed a polynomial FIF. Similarly, if $q_i$  are rational functions, then we call the corresponding FIF as a rational FIF. The parameter $-1< \alpha_i<1$
is called a scaling factor of the transformation $w_i$, and $\alpha = (\alpha_1, \alpha_2, \dots, \alpha_{N-1})$ is the scale vector corresponding to the IFS. The scaling factors provide flexibility in the choice of an interpolant, in contrast to the uniqueness of the interpolant in the traditional methods, and also determine the dimension of the interpolant. To get a rational  FIF with $\mathcal{C}^{p}$-continuity, we need the following proposition appeared in \cite{CVN}. This can be obtained by using Barnsley-Harrington theorem \cite{B2}.
\begin{proposition}\label{HALDIASANGEETAprop2}
Let $\{(x_i, y_i): i \in \mathbb{N}_N\}$ be a given interpolation data with strictly increasing abscissae. Consider the IFS $\mathcal{I}$, where  $L_i(x) = a_i x + b_i$ satisfies (\ref{HALDIASANGEETAeq1}) and
$F_i(x,y)=\alpha_i y + q_i(x),~ q_i(x)=\frac{P_i(x)}{Q_i(x)}$, $P_i(x), Q_i(x)$
are suitably chosen polynomials in $x$ of degree $M, N$ respectively, and
$Q_i(x) \neq 0 $ for every $ x\in [x_1, x_N]$. Suppose for some integer $p \ge 0, |\alpha_i| < a_i^p , i \in \mathbb{N}_{N-1}$. Let
$F_{i,m}(x, y) = \frac{\alpha_i y + q^{(m)}_i(x)}{a_i^m},
q^{(m)}_i$  represents the  $m^{th}$  derivative of $q_i(x)$
with respect to $x$,
\begin{eqnarray*}
y_{1,m} = \frac{q_1^{(m)}(x_1)}{a_1^m - \alpha_1}, \hspace{0.3cm} y_{N,m} =
\frac{q_{N-1}^{(m)}(x_N)}{a_{N-1}^m -
\alpha_{N-1}}, \; m =  1, 2, \dots, p.
\end{eqnarray*}
If $F_{i-1,m}(x_N, y_{N,m})=F_{i,m}(x_1,y_{1,m})$  for $\ i =2,3,
\dots,N-1$ and  $\ m =1,2,\dots,p$, then the IFS $\big\{X;\big(L_i(x),F_i(x,y)\big);i\in \mathbb{N}_{N-1}\big\}$ determines a
rational FIF ${\Psi}\in \mathcal{C}^{p}[x_1,x_N]$, and $\Psi^{(m)}$ is the rational FIF determined by the IFS $\big\{X;\big(L_i(x),F_{i,m}(x,y)\big);i\in \mathbb{N}_{N-1}\big\}$ for $ m =
1,2,\dots,p$.
\end{proposition}
This completes our preparations for the current study, and we are now ready for our main section.
\section{$\mathcal{C}^1$-Rational cubic spline FIF with two-families of shape parameters}\label{HALDIASANGEETAsec3}
Let $\{(x_i,y_i): i\in \mathbb{N}_N\}$, $x_1<x_2<\dots<x_N$, be a
given set of data points. Let $y_i$ and $d_i$ denote the function value and the derivative value at the knot $x_i$, $i\in \mathbb{N}_N$ respectively.
The desired rational cubic spline  FIF with two families of shape
parameters can be obtained by the IFS given in (\ref{HALDIASANGEETAeq5}) with
$$q_i(x)\equiv
q_i^*(\theta)=\frac{U_i (1-\theta)^3 + V_i (1-\theta)^2  \theta+ W_i (1-\theta) \theta^2  + Z_i \theta^3 }{u_i +v_i \theta(1-\theta)},\:
\theta = \frac{x-x_1}{x_N-x_1},\: x \in I.$$
With this special choice of $q_i(x)$, suppose the contraction map  $T$ has a unique fixed point $\Psi \in \mathcal{G}$,  which satisfies:
\begin{equation}\label{HALDIASANGEETAeq6}
\begin{split}
\Psi\big(L_i(x)\big) &= F_i\big(x, \Psi(x)\big)= \alpha_i \Psi(x)+q_i(x),\\
 &= \alpha_i \Psi(x) + \frac{U_i (1-\theta)^3 + V_i (1-\theta)^2  \theta+ W_i (1-\theta) \theta^2  + Z_i \theta^3 }{u_i +v_i \theta(1-\theta)}.
\end{split}
\end{equation}
The conditions $F_i(x_1,y_1)=y_i$, $F_i(x_N,y_N)=y_{i+1}$ can be
reformulated as the interpolation conditions $\Psi(x_i)=y_i$,
$\Psi(x_{i+1})=y_{i+1}$, $ i\in \mathbb{N}_{N-1}$. The interpolatory conditions determine the coefficients $U_i$ and $Z_i$ as follows.
Substituting $x = x_1$ in (\ref{HALDIASANGEETAeq6}) we get
\begin{equation*}
 \Psi\big(L_i(x_1)\big)  = \alpha_i  \Psi(x_1) + U_i
\implies  y_i  = \alpha_i y_1 + \frac{U_i}{u_i}  \implies  U_i
=u_i(y_i-\alpha_i y_1).
\end{equation*}
Similarly, taking $x = x_N$ in (\ref{HALDIASANGEETAeq6}) we obtain $ Z_i  =
u_i(y_{i+1}-\alpha_i y_N)$. \\Now we make $ \Psi\in \mathcal{C}^1(I)$  by imposing the conditions prescribed in Proposition \ref{HALDIASANGEETAprop2}.\\
By hypothesis, $|\alpha_i | \le \kappa a_i$, $ i\in \mathbb{N}_{N-1}$, where $0\le \kappa < 1$. We also have $q_i \in \mathcal{C}^{1}(I)$. Adhering to the notation of Proposition \ref{HALDIASANGEETAprop2}, for $ i\in \mathbb{N}_{N-1}$, we let:\\
\begin{equation*}
F_{i,1}(x, y)= \frac{\alpha_i  y+ q_i^{(1)}(x)}{a_i},
y_{1,1}=d_1,\; y_{N,1}=d_N,\; F_{i,1}(x_1, d_1)=d_i,\;
F_{i,1}(x_N,d_N)=d_{i+1}.
\end{equation*}
Then by Proposition \ref{HALDIASANGEETAprop2}, the FIF $\Psi \in
\mathcal{C}^1(I)$. Further, $\Psi^{(1)}$ is the fractal function
determined by the IFS $\mathcal {I}^*\equiv\big\{\mathbb{R}^2;\big(L_i(x),
F_{i,1}(x,y)\big); i \in \mathbb{N}_{N-1}\big\}$. Consider $\mathcal{G^*} := \{h^*\in \mathcal{C}(I): h^*(x_1)=d_1~\text{and}~h^*(x_N)=d_N\}$ endowed
with the uniform metric. The IFS $\mathcal{I}^*$ induces a
contraction map $T^*: \mathcal{G^*} \rightarrow \mathcal{G^*}$ defined
by $(T^*g^*)\big(L_i(x)\big)=F_{i,1}\big(x, g^*(x)\big),\; x \in I.$
The fixed point of $T^*$ is $\Psi^{(1)}$. Consequently, $\Psi^{(1)}$
satisfies the functional equation:
\begin{equation}\label{HALDIASANGEETAeq7}
 \Psi^{(1)}\big(L_i(x)\big) = F_{i,1}\big(x, \Psi^{(1)}(x)\big)
 = \frac{\alpha_i  \Psi^{(1)}(x) + q_i^{(1)}(x)}{a_i}.
\end{equation}
The  conditions $F_{i,1}(x_1,d_1)=d_i$ and
$F_{i,1}(x_N,d_N)=d_{i+1}$ can be reformulated as the
interpolation conditions for the derivative:
$\Psi^{(1)}(x_i)=d_i$ and $\Psi^{(1)}(x_{i+1})=d_{i+1}$, $ i\in \mathbb{N}_{N-1}$.
Applying $x = x_1$ in (\ref{HALDIASANGEETAeq7}) we obtain
\begin{equation*}
\begin{split}
 \Psi^{(1)}(L_i(x_1)) &  = \frac{\alpha_i}{a_i}  \Psi^{(1)}(x_1) + \frac{u_iV_i - (3u_i+v_i) U_i}{u_i^2 h_i}\\
\implies   V_i & =  (3u_i+v_i)(y_i-\alpha_iy_1)+u_i h_id_i-\alpha_i u_i (x_N-x_1) d_1.
\end{split}
\end{equation*}
Similarly, the substitution  $x = x_N$ in (\ref{HALDIASANGEETAeq7}) yields
\begin{center}
$W_i  = (3u_i+v_i)(y_{i+1}-\alpha_iy_N)-u_i h_id_{i+1}+\alpha_i u_i (x_N-x_1) d_N$.
\end{center}
These values of $U_i, V_i, W_i$, and $Z_i$ reformulate the  desired $\mathcal{C}^1$-rational cubic spline FIF (\ref{HALDIASANGEETAeq6}) to the following:
\begin{equation}\label{HALDIASANGEETAeq8}
 \Psi\big(L_i(x)\big) = \alpha_i \Psi(x) + \frac{P_i(x)} {Q_i(x)},
\end{equation}
$P_i(x)\equiv P_i^* (\theta) =   u_i(y_i-\alpha_i y_1) (1-\theta)^3 + \{(3u_i+v_i)(y_i-\alpha_iy_1)+u_i h_id_i-\alpha_i u_i (x_N-x_1) d_1\} (1-\theta)^2  \theta+
 \{(3u_i+v_i)(y_{i+1}-\alpha_iy_N)-u_i h_id_{i+1}+\alpha_i u_i (x_N-x_1) d_N\} (1-\theta) \theta^2  + u_i(y_{i+1}-\alpha_i y_N) \theta^3,$~~~
$Q_i(x)\equiv Q_i^*(\theta) = u_i +v_i \theta(1-\theta),~ \theta = \frac{x-x_1}{x_N - x_1} $.\\\\
Since the FIF $\Psi$ in (\ref{HALDIASANGEETAeq8}) is derived as a solution of the
fixed point equation $Tg = g$, it is unique for a fixed choice of
the scaling factors and the shape parameters.
\begin{remark}\label{HALDIASANGEETArem1}(Interval tension property) Let $ \triangle_i = \dfrac{y_{i+1} - y_i}{h_i}$.  (\ref{HALDIASANGEETAeq8}) can be expressed as:
\begin{eqnarray}\label{HALDIASANGEETAeq8a}
\Psi(L_i(x)) = &\alpha_i \Psi(x) + (y_i -\alpha_i y_1) (1-\theta)+
(y_{i+1}-\alpha_i y_N)\theta \\& +  \dfrac{u_ih_i \theta (1-\theta)\big [(2\theta-1)\triangle_i^*
+ (1-\theta)d^*_i-\theta d^*_{i+1}\big ]}{Q_i(\theta)} \nonumber,
\end{eqnarray}
where $d^*_i = d_i -  \frac {\alpha_i d_1}{a_i},~
d^*_{i+1} = d_{i+1} -  \frac {\alpha_i d_N}{a_i},~
\triangle^*_i = \triangle_i - \alpha_i \frac{y_N -y_1}{h_i}$.
When $v_i \rightarrow \infty$ in (\ref{HALDIASANGEETAeq8a}),
$\Psi$ converges to the following affine FIF :
\begin{equation}\label{HALDIASANGEETAeq8b}
\Psi(L_i(x)) = \alpha_i \Psi(x) + (y_i -\alpha_i y_1) (1-\theta)
+ (y_{i+1}-\alpha_i y_N)\theta.
\end{equation}
Again if $\alpha_i \rightarrow 0^+$ with $v_i\rightarrow \infty,$
then the rational cubic FIF modifies to  the classical affine interpolant. Hence,
the shape parameter $v_i$ has a vital influence on the graphical
display of data while $u_i$ can assume any positive value. The increase in
the value of parameter $v_i$ in $[x_i,x_{i+1}]$ transforms the rational cubic
functions to the straight line $y_i(1-\theta)+y_{i+1}\theta$.
\end{remark}
\begin{remark}\label{HALDIASANGEETArem2} If $\alpha_i =0$, $ i\in \mathbb{N}_{N-1},$ then the resulting  rational Cubic FIF coincides with the piecewise defined nonrecursive classical  rational cubic interpolant $C$ as
\begin{equation}\label{HALDIASANGEETAeq9}
\Psi(L_i(x)) =\frac{P_i^* (\theta)}{Q_i^* (\theta)},
\end{equation}
where
$P_i^* (\theta) =  u_i y_i(1-\theta)^3+[(3u_i+v_i)y_i+u_i h_i d_i]
(1-\theta)^2 \theta+[(3u_i+v_i)y_{i+1}-u_i h_i d_{i+1}](1-\theta) \theta^2+u_i y_{i+1}\theta^3$,~
$Q_i^*(\theta) = u_i +v_i \theta(1-\theta)$.
Since $\frac {L_i^{-1}(x) -x_1}{x_N-x_1}= \frac{x-x_i}{h_i}=\upsilon$ (say), from (\ref{HALDIASANGEETAeq9}), for $x \in I_i=[x_i, x_{i+1}]$, we have
\begin{equation}\label{HALDIASANGEETAeq10}
%\begin{split}
\Psi(x) =  \frac{P_i^*(\upsilon)}{Q_i^*(\upsilon)} \equiv C_i(x)~ (say).
%\end{split}
\end{equation}
where  $\upsilon$ is a localized variable. The
rational cubic spline $C\in \mathcal{C}^1(I)$ is defined by
$C\big|_{I_i}=C_i$, $i \in \mathbb{N}_{N-1}$. This illustrates that if we let  $\alpha_i\rightarrow 0$, then the graph of our rational cubic FIF on $[x_i,x_{i+1}]$ approaches the graph of the classical rational cubic interpolant described by Sarfraz and Hussain \cite{SHH}.
\end{remark}
\begin{remark}\label{HALDIASANGEETArem3} It is interesting to note that when
$u_i = 1, v_i = 0 ~ \text{and}
~|\alpha_i|\le \kappa a_i $ for $ i\in \mathbb{N}_{N-1}$, $\kappa \in (0,1)$, in Eqn. (\ref{HALDIASANGEETAeq8}) then the resulting RCFIF coincides with the $\mathcal{C}^1$-cubic Hermite FIF \cite{CV2}. If we take $ u_i =1, v_i = 0 ~ \text{and}
~\alpha_i=0$, we obtain for $ x \in [x_i,x_{i+1}]$,\\
$\Psi(x) = (2 \theta^3-3\theta^2+1) y_i + (\theta^3-2\theta^2+\theta)h_i d_i+(-2 \theta^3+3\theta^2)y_{i+1}+(\theta^3-\theta^2)\cdot\\h_i d_{i+1}.$
Hence $ \Psi$ recovers the classical piecewise $\mathcal{C}^1$-cubic Hermite interpolant over $I$.
\end{remark}
\section{Parameter Identification of the Rational FIF}\label{HALDIASANGEETAsec4}
In this section, we take up the problem of identifying the parameters of the rational FIF so that the corresponding $\mathcal{C}^1$-RCSFIF enjoys certain desirable shape properties.  In section \ref{HALDIASANGEETAsubsec4a}, we derive sufficient condition on the parameters so that the RCSFIF generates positive curves for a given positive data set. In the next subsection \ref{HALDIASANGEETAsubsec4b},
we deal with a slightly general problem, viz., finding conditions on the IFS parameters for the containment of the graph of the corresponding FIF within a prescribed rectangle. We identify suitable
values of the parameters that render RCSFIFs with graphs lying above a prescribed
straight line in section \ref{HALDIASANGEETAsubsec4c}.
\subsection{Positivity Preserving $\mathcal{C}^1$-Rational Cubic Spline
FIF}\label{HALDIASANGEETAsubsec4a}
For an arbitrary selection of the scaling factors and shape parameters,
the RCSFIF  $\Psi$ described above may not be positive, if the
data set is positive. This is very similar to the ordinary spline schemes that do not provide the desired shape features of data. We have to restrict the scaling factors and shape parameters for a positive preserving rational cubic  FIF as described in the following theorem:
\begin{theorem}\label{RCFIFthm}
Let $\{(x_i ,y_i ,d_i),i \in \mathbb{N}_N \}$ be a given positive data. If
$(i)$ the scaling factor $\alpha_i$ is selected as\\
\begin{equation}\label{HALDIASANGEETAeq11}
\alpha_i \in [0,\min\{a_i ,\frac{y_i}{y_1} ,\frac{y_{i+1}}{y_N}\}),i \in \mathbb{N}_{N-1},\\
\end{equation}
$(ii)$ For $i \in \mathbb{N}_{N-1}$, the shape parameters $u_i,~ \text{and}
~v_i,$ are restricted as
%\begin{centre}
\begin{equation}\label{HALDIASANGEETAeq12}
\begin{split}
&u_i > 0 ,\;\;\\
&v_i\geq \max \Big\{0,-u_i\big[3+\frac{h_i d_i-\alpha_i (x_N-x_1)d_1}{y_i -\alpha_i y_1}\big],
-u_i\big[3+\frac{(-h_i d_{i+1}+\alpha_i (x_N-x_1)d_N)}{y_{i+1} -\alpha_i y_N}\big]\Big\},\\
\end{split}
\end{equation}
%\end{centre}
then we obtain a positivity preserving $\mathcal{C}^1$-RCSFIF $\Psi$.
\end{theorem}
\begin{proof}
We have
$\Psi(L_i(x)) = \alpha_i\Psi(x)+  \frac{P_i^*(\theta)}{Q_i^*(\theta)}$,
$ \theta$ = $ \frac{x-x_1}{x_{N}-x_{1}}$, $ x \in I.$
If $\psi(x)\geq 0,$ it is easy to verify that for $i \in \mathbb{N}_{N-1}$ when $\alpha_i\geq0,i \in \mathbb{N}_{N-1}$,
the sufficient condition for $\Psi(L_i(x))>0$ for all $ x \in I$ is
$\frac{P_i^*(\theta)}{Q_i^*(\theta)} >0 \;\;\forall\;\; \theta \in [0,1].$
So the initial condition on the scaling factors are $\alpha_i\geq0, i\in \mathbb{N}_{N-1}$.
Since $ Q_i^*(\theta)>0\;\;\forall\;\;\theta \in [0,1]$, $u_i>0$, and $v_i>0$, $ i \in \mathbb{N}_{N-1} $, the positivity of RCSFIF
$\frac{P_i^*(\theta)}{Q_i^*(\theta)} $ depends on the positivity of $ P_i(\theta)$.
 For our convenience, we write  $ P_i^*(\theta)$ as
$P_i^*(\theta) = U_i (1-\theta)^3 + V_i (1-\theta)^2  \theta + W_i (1-\theta) \theta^2 + Z_i \theta^3.$ Now $P_i^*(\theta) > 0$ if $ U_i>0 , V_i>0 ,W_i>0 , Z_i>0$ (cf. Section 3).
Now $ U_i>0, Z_i>0$ if (\ref{HALDIASANGEETAeq11}) is true. Also $V_i>0, W_i>0$ when (\ref{HALDIASANGEETAeq12}) is valid.
Since $\Psi(x)$ is constructed iteratively
$\Psi(L_i(.))\geq 0\;\;\forall\;\;\psi(.)\geq 0$ implies $\Psi(x)\geq0;\;\forall\;\; x\in I.$
\end{proof}
\noindent
\textbf{Consequence.} Setting $\alpha_i=0$  in (\ref{HALDIASANGEETAeq12}), we obtain
$u_i > 0$,
$v_i\geq \max \{0,-u_i[3+\frac{h_i d_i}{y_i}],
-u_i[3+\frac{(-h_i d_{i+1})}{y_{i+1}}]\},$
that provide a set of sufficient conditions for the positivity of
rational cubic spline $C$. It is worthwhile to mention that these
conditions on shape parameters are weaker than those obtained by Sarfraz et.al.  \cite{SHH}.
\subsection{Containment of  Graph of  FIF in a Rectangle}\label{HALDIASANGEETAsubsec4b}
Given a data $\{(x_i, y_i): i\in \mathbb{N}_N\}$. We wish to
find conditions on the parameters of the rational IFS so that the
graph $G$ of the corresponding FIF $\Psi$ lies within the prescribed
rectangle $K= I \times [c,d]$, where $c< \min\{y_i:
i\in \mathbb{N}_N\}$ and $d> \max\{y_i: i\in \mathbb{N}_N\}$. Here
graph $G$ is the attractor of the IFS
$\mathcal{I}=\Big\{\mathbb{R}^2;\big(L_i(x), F_i(x,y)=\alpha_i y+ q_i(x)\big);
i\in \mathbb{N}_{N-1}\Big\}$.  The FIF $\Psi$ is obtained by iterating the IFS
$\mathcal{I}$ and $I=[x_1, x_N]$ is the attractor of the IFS $\{I;
L_i(x); i\in \mathbb{N}_{N-1}\}$. Hence, for $G$ to lie within $K$, it suffices
to prove that $ \alpha_i y+ q_i(x) \in [c,d]~ \forall~ (x,y) \in
K$.\\
\noindent\textbf{Case-I} As in the  positivity preserving interpolation discussed
earlier, firstly we assume $ 0 \le \alpha_i <a_i$ for all $i\in
\mathbb{N}_{N-1}$. Let $(x,y) \in K$. Then, with our assumption on $\alpha_i$,
we have $\alpha_i c \le \alpha_i y \le \alpha_i d $. This implies
$\alpha_i c+ q_i(x) \le \alpha_i y+ q_i(x) \le \alpha_i d +
q_i(x)$. Consequently, for $G$ to lie within $K$, it suffices to
have the following conditions for all $i\in \mathbb{N}_{N-1}:$
\begin{equation}\label{HALDIASANGEETAeq13}
c \le \alpha_i c + q_i(x),
\end{equation}
\begin{equation}\label{HALDIASANGEETAeq14}
 \alpha_i d + q_i(x) \le d.
\end{equation}
Now, (\ref{HALDIASANGEETAeq13}) holds if
\begin{equation}\label{HALDIASANGEETAeq15}
c(1-\alpha_i) \le \frac{U_i (1-\theta)^3+ V_i
\theta(1-\theta)^2+ W_i \theta^2 (1-\theta) + Z_i
\theta^3}{u_i +v_i \theta(1-\theta)},
\end{equation}
where the constants $U_i , V_i, W_i$, and $ Z_i$ are
given in (\ref{HALDIASANGEETAeq8}). The expression in the denominator of $q_i(x)$ can be written in the degree elevated form as following:
\begin{equation}\label{HALDIASANGEETAeq15a}
\begin{split}
u_i +v_i \theta(1-\theta) &\equiv u_i(1-\theta)^2+(2u_i+v_i)(1-\theta)\theta+u_i  \theta^2\\&  \equiv u_i (1-\theta)^3 + (3u_i+v_i)(1-\theta)^2 \theta + (3u_i+v_i) (1-\theta) \theta^2 +  u_i \theta^3.
\end{split}
\end{equation}
To keep symmetry of the expressions in the numerator and the denominator, cross multiplying and rearranging it can be observed that (\ref{HALDIASANGEETAeq15})
holds if
\begin{equation}\label{HALDIASANGEETAeq16}
\begin{split}
&\big[U_i -c (1-\alpha_i)u_i\big](1-\theta)^3 + \big[V_i-c
(1-\alpha_i)(3u_i+v_i)\big] \theta
(1-\theta)^2\\&+\big[ W_i-c (1-\alpha_i)(3u_i+v_i)\big]
\theta^2 (1-\theta)+ \big[ Z_i-c(1-\alpha_i) u_i\big] \theta^3
\ge 0.
\end{split}
\end{equation}
Now (\ref{HALDIASANGEETAeq16}) is satisfied if the following system of
inequalities are true:
\begin{equation*}\label{HALDIASANGEETAeq17}\left.
\begin{split}
U_i -c (1-\alpha_i)u_i & \ge 0,\;\; V_i-c
(1-\alpha_i)(3u_i+v_i) &\ge 0,\\ W_i-c (1-\alpha_i)(3u_i+v_i) &\ge 0,\;\;~~~~~~~~~~~~~ Z_i-c(1-\alpha_i) u_i
&\ge 0.
\end{split}\right\}
\end{equation*}
Since $u_i>0$, the selection of $\alpha_i$ satisfying
$$\alpha_i \le \dfrac{y_i-c}{y_1-c}\;\;  \text{and}\; \; \alpha_i \le
\dfrac{y_{i+1}-c}{y_N-c}$$ ensures $U_i -c (1-\alpha_i)\alpha_i \ge
0\;\; \text{and}\;\; Z_i-c(1-\alpha_i) \delta_i\ge 0$. Let
$\alpha_i < \min \big\{\dfrac{y_i-c}{y_1-c},
\dfrac{y_{i+1}-c}{y_N-c}\big\}$.\\ Then, the conditions $V_i-c
(1-\alpha_i)(3u_i+v_i)\ \ge 0,\; W_i-c (1-\alpha_i)(3u_i+v_i) \ge 0$ are met if
\begin{equation*}
\begin{split}
v_i & \ge v_{1i}:= -u_i\big[3+\frac{h_i d_i-\alpha_i (x_N-x_1)d_1}{y_i-c-\alpha_i(y_1-c)}\big],\\
v_i &\ge
v_{2i}:= -u_i\big[3+\frac{(-h_i d_{i+1}+\alpha_i (x_N-x_1)d_N)}{y_{i+1}-c-\alpha_i(y_N-c)}\big].
\end{split}
\end{equation*}
Hence, the following conditions are sufficient to verify
(\ref{HALDIASANGEETAeq13}):
\begin{equation*}\label{HALDIASANGEETAeq18}
0 \le \alpha_i  < \min\Big\{\frac{y_i-c}{y_1-c},
\frac{y_{i+1}-c}{y_N-c}\Big\},\;\; u_i\ge 0, \;\;  v_i \ge \max \{v_{1i},v_{2i}\}.
\end{equation*}
Similarly, the following conditions verify (\ref{HALDIASANGEETAeq14}):
\begin{equation*}\label{HALDIASANGEETAeq19}
\begin{split}
0 \le \alpha_i  &< \min\Big\{\frac{d-y_i}{d-y_1},
\frac{d-y_{i+1}}{d-y_N}\Big\},\\
v_i & \ge v_{3i}:= -u_i\big[3+\frac{(-h_i d_i+\alpha_i (x_N-x_1)d_1)}{d-y_i-\alpha_i(d-y_1)}\big],\\
v_i &\ge
v_{4i}:= -u_i\big[3+\frac{(h_i d_{i+1}-\alpha_i (x_N-x_1)d_N)}{d-y_{i+1}-\alpha_i(d-y_N)}\big].
\end{split}
\end{equation*}
\noindent\textbf{Case-II} We consider $-a_i< \alpha_i <0$. In this case (\ref{HALDIASANGEETAeq13}) and
(\ref{HALDIASANGEETAeq14}) will be replaced respectively by
\begin{equation}\label{HALDIASANGEETAeq20}
c \le \alpha_i d+ q_i(x) ~~\text{and}~~  \alpha_i c + q_i(x) \le
d.
\end{equation}
Using computations similar to the case of $\alpha_i \ge 0$, it can
be seen that  (\ref{HALDIASANGEETAeq20}) is true if
\begin{equation*}\label{HALDIASANGEETAeq21}
\begin{split}
&\alpha_i^{min}  >  \max \Big\{-a_i, \frac{y_i-c}{y_1-d},
\frac{y_{i+1}-c}{y_N-d}, \frac{d-y_i}{c-y_1}, \frac{d-y_{i+1}}{c- y_N} \Big\},\\
&v_i \ge \max \Big \{v_{5i}:=-u_i\big[3+\frac{(h_i d_i-\alpha_i (x_N-x_1)d_1)}{y_i-c-\alpha_i(y_1-d)}\big],v_{6i}:=-u_i\big[3+\frac{(-h_i d_{i+1}+\alpha_i (x_N-x_1)d_N)}{y_{i+1}-c-\alpha_i(y_N-d)}\big],\\&v_{7i}:=-u_i\big[3+\frac{(-h_i d_i+\alpha_i (x_N-x_1)d_1)}{d-y_i-\alpha_i(c-y_1)}\big],v_{8i}:=-u_i\big[3+\frac{(h_i d_{i+1}-\alpha_i (x_N-x_1)d_N)}{d-y_{i+1}-\alpha_i(c-y_N)}\big]\Big\}.
\end{split}
\end{equation*}
The following theorem contains the above discussion.
\begin{theorem}\label{r2pptm5}
Suppose a data set $\{(x_i,y_i):i\in \mathbb{N}_N\}$ is given, and
$\Psi$ is the corresponding $\mathcal{C}^1$-rational cubic spline FIF
described in (\ref{HALDIASANGEETAeq8}). Then the following conditions on the
scaling factors and the shape parameters in each subinterval are
sufficient for the  containment of the graph of $\Psi$ in the
rectangle $I= [x_1, x_N] \times [c,d]:$
\begin{equation*}
\begin{split}
\alpha_i ^{min} < \alpha_i  < \alpha_i ^{max},~u_i >0,~\text{and}~ v_i > \max \{0,v_{1i}, v_{2i},v_{3i},v_{4i},v_{5i},v_{6i},v_{7i},v_{8i}\},
\end{split}
\end{equation*}
where $\alpha_i ^{max}< \min\Big\{a_i,\frac{y_i-c}{y_1-c},
\frac{y_{i+1}-c}{y_N-c},\frac{d-y_i}{d-y_1},
\frac{d-y_{i+1}}{d-y_N}\Big\}.$
\end{theorem}
\begin{remark}\label{r2pprm4b}
The positivity preserving interpolation discussed in Theorem 1 can be obtained as a special case of our present setting. The idea is to consider that the rectangle is enlarged to the semi-infinite strip $K= I \times (0, \infty)$.  In this case,
no condition need be imposed for bounding the graph from above.
Thus, we have to consider only (\ref{HALDIASANGEETAeq13}) with the special choice
$c=0$. Adhering to the above notation and discussion, we get the
desired conditions for the positivity as discussed
in Theorem 1.
\end{remark}
\subsection{Rational Cubic Spline FIF above the line}\label{HALDIASANGEETAsubsec4c}
\begin{theorem}\label{r2pptm5}
Suppose a data set $\{(x_i,y_i):i\in \mathbb{N}_N\}$ which lies above the line $t=mx+k$, that is $y_i>t_i$ for $i\in \mathbb{N}_N$. The graph of the corresponding RCSFIF $\Psi$ lies above the line $t=mx+k$ if the scaling factors $|\alpha_i|<a_i$ and the shape parameters $u_i >0, v_i >0$ further satisfy
\begin{equation*}
\begin{split}
0\le \alpha_i < \min \big\{\dfrac{y_i-t_i}{y_1-t_1},
\dfrac{y_{i+1}-t_{i+1}}{y_N-t_N}\big\} ,~u_i >0,~\text{and}~ v_i > \max \{0,v_{9i}, v_{10i}\},
\end{split}
\end{equation*}
where  $v_i \ge v_{9i}:= -u_i\big[\frac{2(y_i-t_i)+(y_i-t_{i+1})+h_i d_i
-\alpha_i \{2(y_1-t_1)+(y_1-t_N)+(x_N-x_1)d_1\}}{y_i-t_i-\alpha_i(y_1-t_1)}\big],\\
v_i \ge
v_{10i}:= -u_i\big[\frac{(y_{i+1}-t_i)+2(y_{i+1}-t_{i+1})-h_i d_{i+1}
-\alpha_i \{(y_N-t_1)+2(y_N-t_N)-(x_N-x_1)d_N\}}{y_{i+1}-t_{i+1}-\alpha_i(y_N-t_N)}\big]$.
\end{theorem}
\begin{proof}
Since $\Psi(x_i) =y_i> t_i$ for all $i\in \mathbb{N}_N$. Therefore to prove $\Psi(\tau)> m \tau+k$ for all $\tau\in I$, it is sufficient to verify that $\Psi(x)> m x+k, x\in I$ implies $\Psi(L_i(x)) > m L_i(x)+k$ for $x\in I$ and for all $i\in \mathbb{N}_{N-1}$. Assume $\Psi(x)> m x+k$. We need to make sure that
\begin{equation}\label{HALDIASANGEETAeq22}
\alpha_i \Psi(x)+\frac{P_i^*(\theta)}{Q_i^*(\theta)}>m(a_ix+b_i)+k.
\end{equation}
We shall impel the scaling parameters such that $0\le\alpha_i<a_i$ for all $i\in \mathbb{N}_{N-1}$. Since $Q_i^*(\theta)>0$ and keeping in mind the assumptions $\Psi(x)> m x+k$, cross multiplying and rearranging it can be observed that (\ref{HALDIASANGEETAeq22}) holds if
\begin{equation}\label{HALDIASANGEETAeq23}
\alpha_i(m x+k)Q_i^*(\theta)+ P_i^*(\theta)- (ma_ix+mb_i+k)Q_i^*(\theta)>0, \theta\in[0,1].
\end{equation}
Substituting $x=x_1+\theta(x_N-x_1)$, the expression for $P_i^*(\theta)$ and using the degree elevated form of $Q_i^*(\theta)$ from Eqn. (\ref{HALDIASANGEETAeq15a}). After some rearrangement, (\ref{HALDIASANGEETAeq23}) reduces to
 \begin{equation}\label{HALDIASANGEETAeq24}
U_i^* (1-\theta)^3 + V_i^*\theta(1-\theta)^2 + W_i^* \theta^2 (1-\theta) +  Z_i^* \theta^3 >0, \theta \in [0,1],
\end{equation}
where\\
\begin{equation*}
\begin{split}
U_i^*&= u_i[y_i-t_i-\alpha_i(y_1-t_1)],\\
V_i^*&= u_i[2(y_i-t_i)+(y_i-t_{i+1})+h_i d_i-\alpha_i\{2(y_1-t_1)+(y_1-t_N)+(x_N-x_1)d_1\}]\\&+v_i[y_i-t_i-\alpha_i(y_1-t_1)],\\
W_i^*&= u_i[(y_{i+1}-t_i)+2(y_{i+1}-t_{i+1})-h_i d_{i+1}-\alpha_i\{(y_N-t_1)+2(y_N-t_N)\\&-(x_N-x_1)d_N\}]+v_i[y_{i+1}-t_{i+1}-\alpha_i(y_N-t_N)],\\
Z_i^*&= u_i[y_{i+1}-t_{i+1}-\alpha_i(y_N-t_N)].
\end{split}
\end{equation*}
With the substitution $\theta=\frac{\nu}{\nu+1}$,  (\ref{HALDIASANGEETAeq24}) is equivalent to $Z_i^*\nu^3 +W_i^* \nu^2 + V_i^* \nu + U_i^*>0$ for all $\nu>0$.
We know that \cite{SH1}, a cubic polynomial $\rho(\xi)=a \xi^3 +b \xi^2+c \xi+d \ge 0$ for all $\xi
\ge0$, if and only if $(a, b, c, d) \in R_1 \cup R_2$, where
\begin{equation*}
\begin{split}
&R_1=\{(a, b, c, d):a\geq0, b\geq0, c\geq0, d\geq0\},\\
&R_2=\{(a, b, c, d): a\geq0, d\geq0, 4ac^3+4db^3+27a^2 d^2-18 a b c
d-b^2c^2\geq0\}.
\end{split}
\end{equation*}
As the condition involved in $R_2$ is computationally cumbersome, to obtain a set of sufficient condition for the positivity $Z_i^*\nu^3 +W_i^* \nu^2 + V_i^* \nu + U_i^*>0$, we use comparatively efficient and  reasonably acceptable choice of parameters determined by $R_1$. Thus the polynomial in (\ref{HALDIASANGEETAeq24}) is positive if $U_i^*> 0, V_i^*> 0, W_i^*> 0$ and $Z_i^*> 0$ are satisfied. It is straight forward to see that $U_i^*> 0$ is satisfied if $\alpha_i<\frac{y_i-t_i}{y_1-t_1}$ and $Z_i^*> 0$ if $\alpha_i<\frac{y_{i+1}-t_{i+1}}{y_N-t_N}$. It is plain to see that the additional conditions on the shape parameters $u_i>0$ and $v_i>0$ prescribed in the theorem ensure the positivity of  $V_i^*$ and $W_i^*$. This completes the proof.
\end{proof}
\begin{remark}\label{r2pprm4a}
In a similar fashion one can deal with the problem of curve lies below the line. Taking $m=k=0$ in the preceding theorem, we reobtain the positivity of RCSFIF studied elaborately in Theorem 1. If assume $\alpha_i=0$ for all $i\in \mathbb{N}_{N-1}$, which provide sufficient condition for the traditional rational cubic spline to lie above the line, see \cite{SHH}.
\end{remark}
\section{Convergence  Analysis of Rational Cubic Spline FIFs}\label{HALDIASANGEETAsec5}
In this section, the uniform error bound for a RCSFIF $\Psi$ is obtained from the Hermite data  $\{(x_i,y_i,d_i): i \in \mathbb{N}_N\}$ satisfying $x_1<x_2<\dots<x_N$, being interpolated and generated from a function $\Phi \in\mathcal{C}^3(I)$. By using $\|\Phi- \Psi\|_{\infty} \le \|\Phi - C\|_{\infty} + \|C - \Psi\|_{\infty}$, we will derive the convergence of $\Psi$ to the original function $\Phi$ using the convergence results for its classical counterpart $C$ and the uniform distance between $\Psi$ and $C$. The first summand in the above inequality is obtained from Theorem 7.1 of \cite{SHH} as $\|\Phi-C\|_\infty \leq \frac{1}{2}\|\Phi^{(3)}\|_\infty \underset{1\leq i\leq {N-1}}\max \{h_i^3 c_i\}$, for some suitable constant $c_i$ independent of $h_i$. The rightmost summand is obtained by using the definition of the Read-Bajraktarevi\'{c} operators for which $ \Psi $ is a fixed point and by applying the Mean Value Theorem. To make our presentation simple, we introduce the following notations: $|y|_{\infty}
 = \max \{|y_i| :i \in \mathbb{N}_N\} $, $ |d|_{\infty} = \max \{|d_i| : i \in \mathbb{N}_N \} $, $ |u|_{\infty} = \max \{
|u_i| : i \in \mathbb{N}_{N-1} \}$, $ |v|_{\infty} = \max \{
|v_i| : i \in \mathbb{N}_{N-1} \}$, $|\alpha|_\infty=\max\{|\alpha_i|: i \in \mathbb{N}_{N-1}\}$, $ h = \max \{h_i : i \in \mathbb{N}_{N-1} \} $. The proof is just consequent upon strictly routine matter of simple calculations.
\begin{theorem}\label{RHFIFthm2}
Let $\Phi\in \mathcal{C}^3(I)$ be the original function, $\Psi$ and
$C$, respectively, be the RCSFIF and the classical
rational cubic interpolant for $\Phi$ with respect to the
interpolation data $\{(x_i,y_i,d_i) : i= 1,2,\dots,N\},
y_i=\Phi(x_i)$. Then,
\begin{eqnarray*}
\|\Phi-S\|_\infty &\leq \frac{1}{2}\|\Phi^{(3)}\|_\infty h^3 c +\frac{|\alpha|_\infty}{s(1-|\alpha|_\infty)}\Big\{|u|_\infty M+
\frac{1}{4}\big[(3|u|_\infty+|v|_\infty)M+|u|_\infty (h
|d|_\infty\\&+(x_N-x_1)\max\{|d_1|,|d_N|\})\big]\Big\},
\end{eqnarray*}
where $M=|y|_\infty+ \max\{|y_1|,|y_N|\}$,   $ s = \min\{ s_i : i \in \mathbb{N}_{N-1} \}$ with $s_i =u_i+\frac{1}{4}v_i$.
\end{theorem}
\section{Numerical Examples}\label{HALDIASANGEETAsec6}
Consider the positive data set $\{(x_i,y_i)\}$=$\{(0,0.1), (0.4,1), (0.75,2), (1,5\}$. The derivative values at the knot points are estimated by using the arithmetic mean method (see \cite{GD85}) as $d_1=-1.5238,d_2=1.5238,  d_3=8.1905,  d_4=15.8095$. Our scheme has been implemented according to the choice of parameters given in Table 1. In Fig. 1(a), we do not follow the prescription given in Theorem 1 for which we obtain a nonpositive RCSFIF.  Fig. 1(b) represents a positive RCSFIF obtained through a choice of parameters as described in Theorem 1. It is noticed that the perturbation in  $\alpha_1$ affects the RCSFIF considerably in the first subinterval, namely $[x_1,x_2]$. By setting all the scaling factors to be zero, the classical rational cubic spline is obtained (see Fig. 1(c)). We calculate the bounds of scaling factors and shape parameters using Theorem 2 so that RCSFIF lying within the rectangle $[0,1]\times[0.1,5]$. Our choices of the scaling factors and shape parameter values are displayed in 
Table 1 and the corresponding RCSFIFs are generated in Figs. 1(d)-(e).\\
\noindent Consider the Hermite data set $\{(1,-1.2,0.85), (3.3,-1.1,-0.15), (4.6,-1,\\-0.4583), (7.2,4.5,-0.7861\}$, which lies above the line $y=-0.5x-1$. By choosing the scaling factors $|\alpha_i|<a_i$ and the shape parameters  $u_i>0$, $v_i>0$, we obtain the RCSFIF in Fig. 1(f) which does not follows the requirement to be above the line  $y=-0.5x-1$. This illustrates the importance of Theorem 3. Choosing the scaling factors and the shape parameters according to the prescription given in Theorem 3 (see Table 1), RCSFIF that lie above the line $y=-0.5x-1$ are shown in Fig. 1(g)-(i). We can notice the effects in the shape of the RCSFIF in the second subinterval due to changes in the scaling factor by $\alpha_2=0.1$ by comparing the RCSFIFs in Figs. 1(g)-(i). Same can be done for the shape parameters also. A classical rational cubic spline above the line $y=-0.5x-1$ is obtained by setting all the scaling factors to be zero as the prescription given in Remark 5.
\begin{center}
\begin{table}[h!]
\caption{Parameters for the rational cubic spline FIFs in Fig 1(a)-(f).}\label{table:surfacepatch_data}
\begin{center}
\scriptsize {
\begin{tabular}{|l |l| l| l|}\hline
$\hspace{0.2cm}Figure \hspace{0.6cm}$&$\hspace{0.2cm}Scaling factors  \hspace{0.6cm}$&$\hspace{0.2cm}Shape parameters\hspace{0.6cm}$\\ \hline
 ~~~1(a)&$\alpha=(-0.2,0.31,0.23)$&$u=(0.1,0.1,0.1)$,$v=(0.08,0.1,0.1)$\\\hline
 ~~~1(b)&$\alpha=(0.2,0.31,0.23)$&$u=(0.1,0.1,0.1)$,$v=(0.08,0.1,0.1)$\\ \hline
 ~~~1(c)&$\alpha=(0,0,0)$&$u=(0.1,0.1,0.1)$,$v=(0.08,0.1,0.1)$ \\ \hline
 ~~~1(d)&$\alpha=(0.1, 0.3,  0.2)$&$u=(0.1,0.1,0.1)$,$v=(0.26,0.1,0.1)$\\\hline
 ~~~1(e)&$\alpha=(0, 0,  0)$&$u=(0.1,0.1,0.1)$,$v=(3.8,0.1,0.1)$\\ \hline
 ~~~1(f)&$\alpha=(-0.3, -0.2,  -0.4)$&$u=(0.1,0.1,0.1)$,$v=(3.8,0.1,0.1)$\\\hline
 ~~~1(g)&$\alpha=(0.17, 0.2,  0.4)$&$u=(0.1,0.1,0.1)$,$v=(3.8,0.1,0.1)$\\ \hline
 ~~~1(h)&$\alpha=(0.17,0.1,0.4)$&$u=(0.1,0.1,0.1)$,$v=(3.8,0.1,0.1)$ \\ \hline
 ~~~1(i)&$\alpha=(0,0,0)$&$u=(0.1,0.1,0.1)$,$v=(3.8,0.1,0.1)$ \\ \hline
\end{tabular}}
\end{center}
\end{table}
\end{center}

%%%%begining of figures
\begin{figure}[h!]\label{r2ppfig2}
\begin{center}
\begin{minipage}{0.32\textwidth}
\epsfig{file=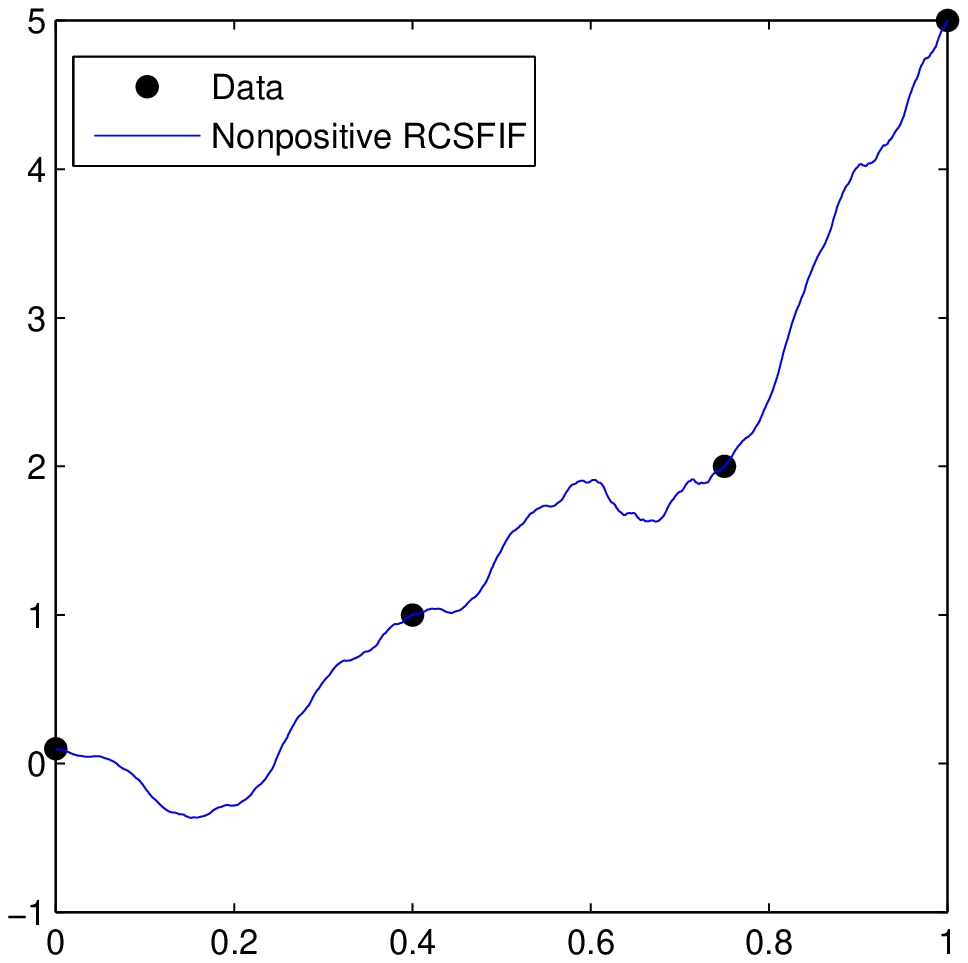,scale=0.29} \centering{\\(a): Nonpositive RCSFIF}
\end{minipage}\hfill
\begin{minipage}{0.32\textwidth}
\epsfig{file=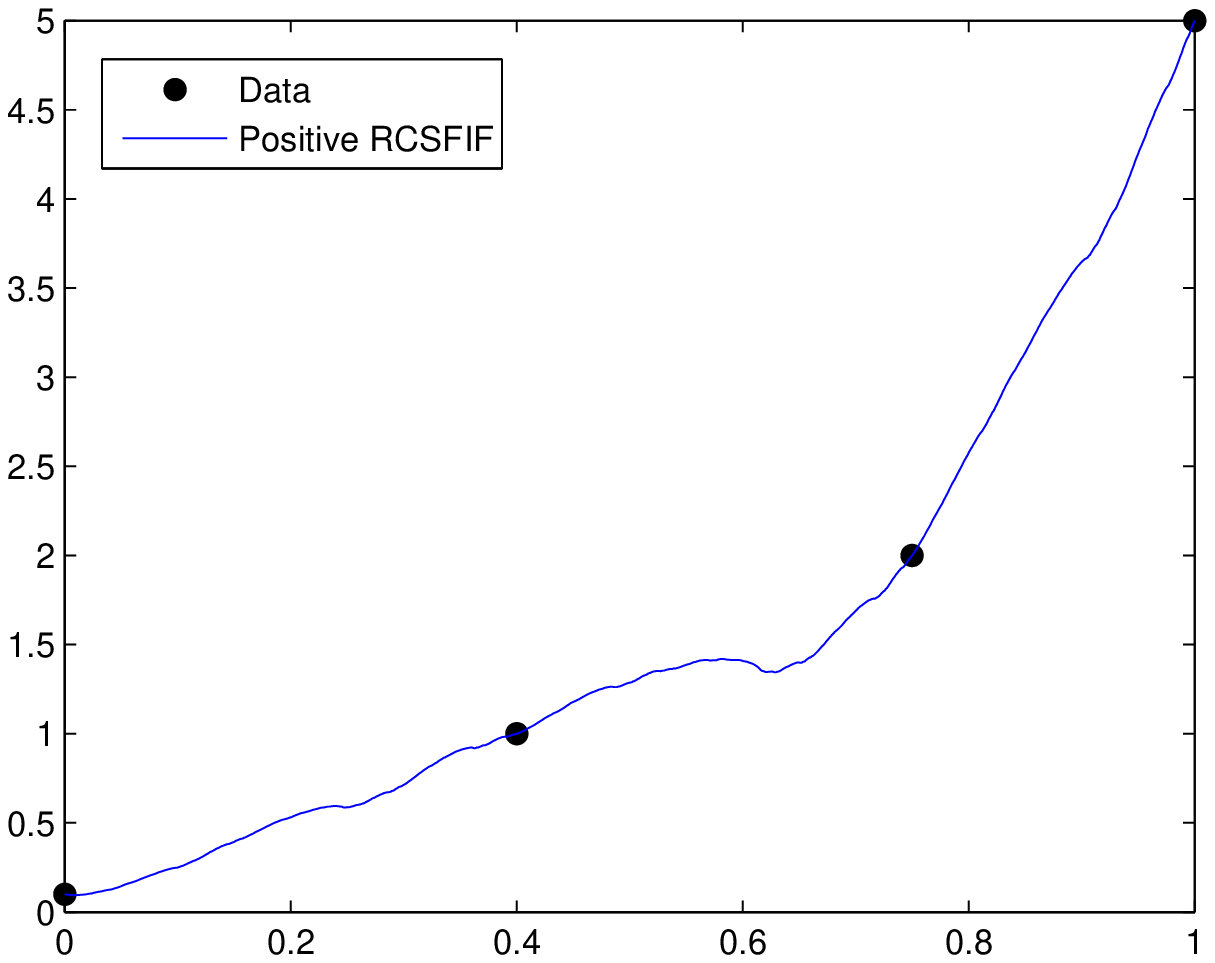,scale=0.29} \centering{\\(b): Positive RCSFIF}
\end{minipage}\hfill
\begin{minipage}{0.32\textwidth}
\epsfig{file =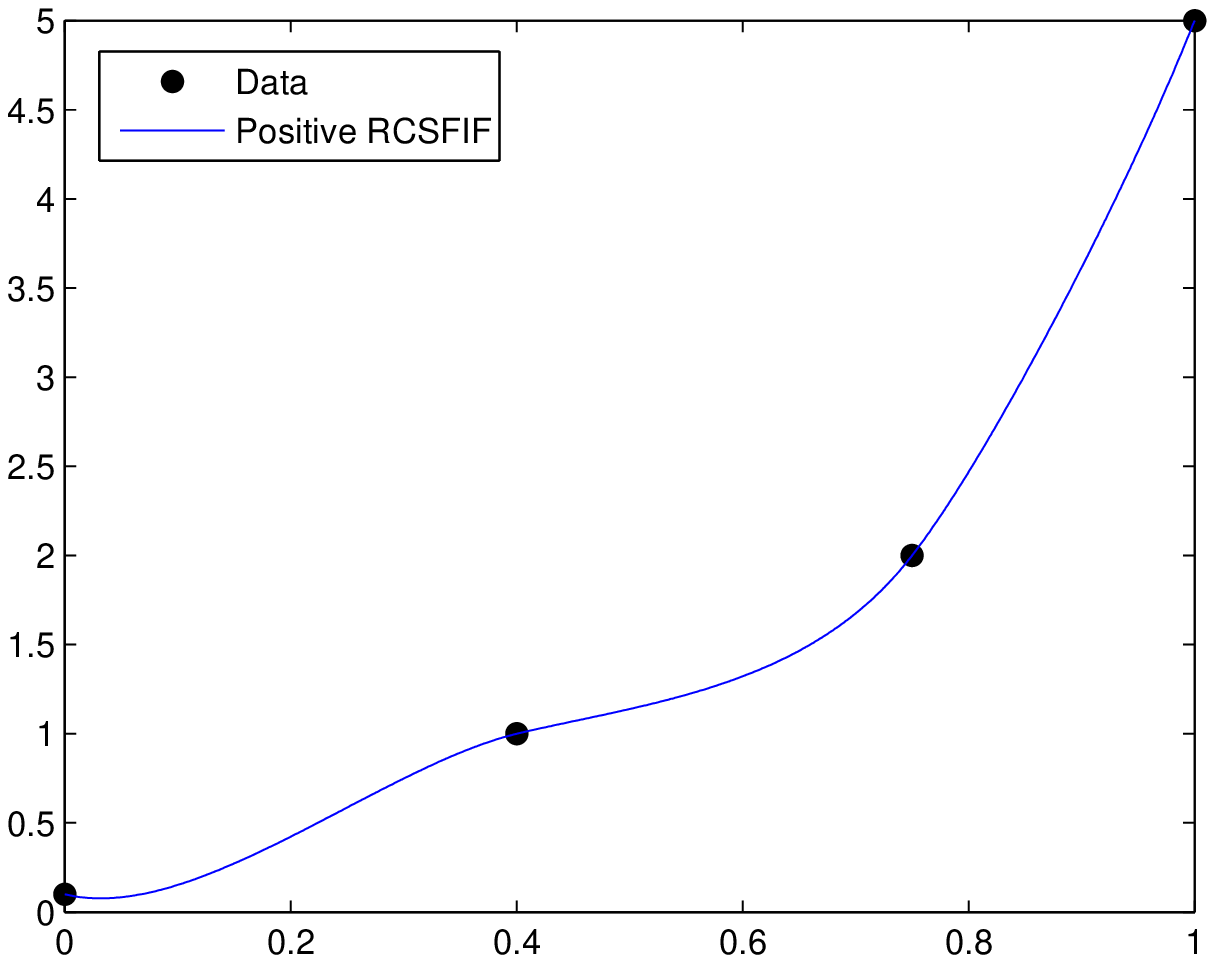,scale=0.29} \centering{\\(c): Classical Positive Rational Cubic Spline}
\end{minipage}\hfill
\begin{minipage}{0.32\textwidth}
\epsfig{file =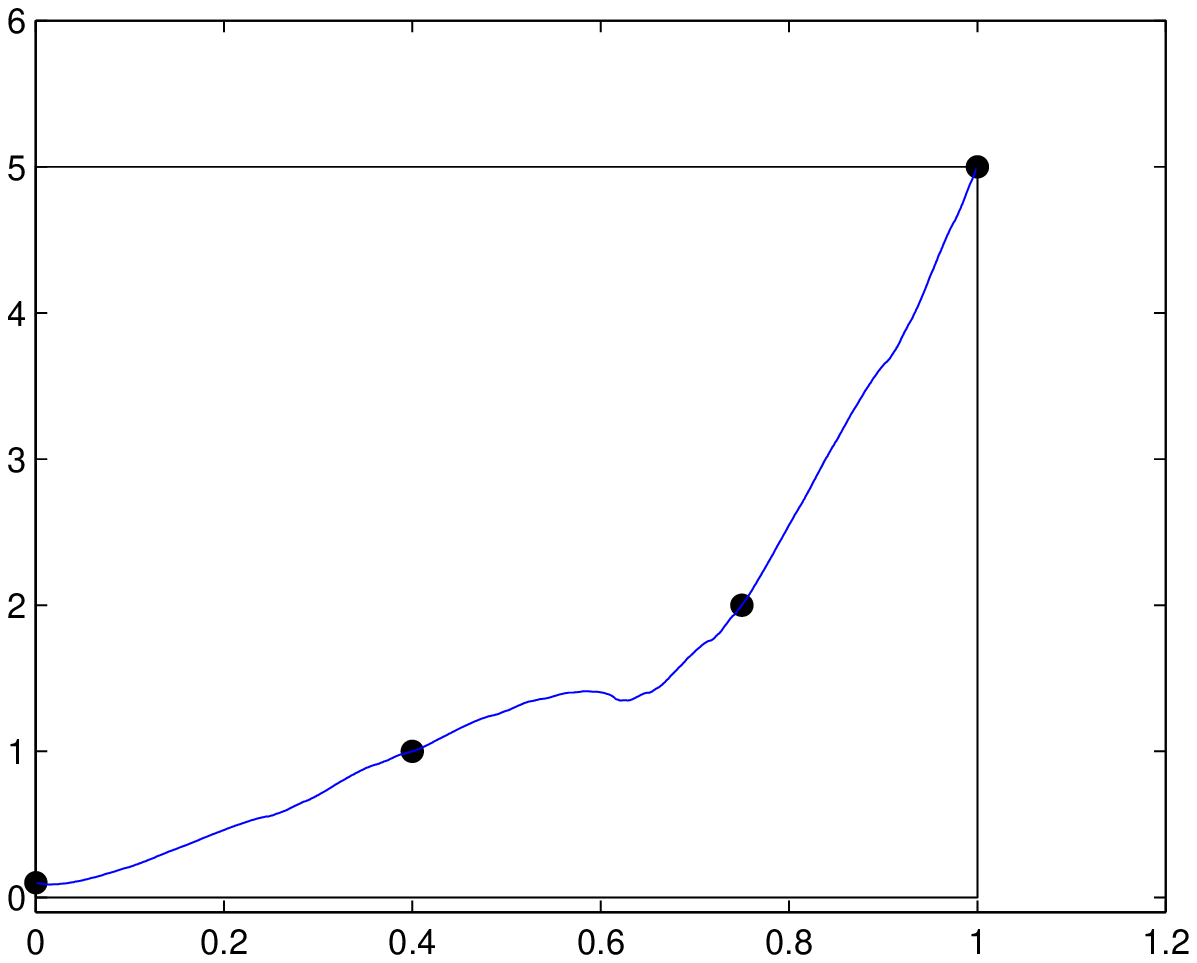,scale=0.29} \centering{\\(d): Rational FIF within rectangle $[0,1]\times[0.1,5]$}
\end{minipage}\hfill
\begin{minipage}{0.32\textwidth}
\epsfig{file=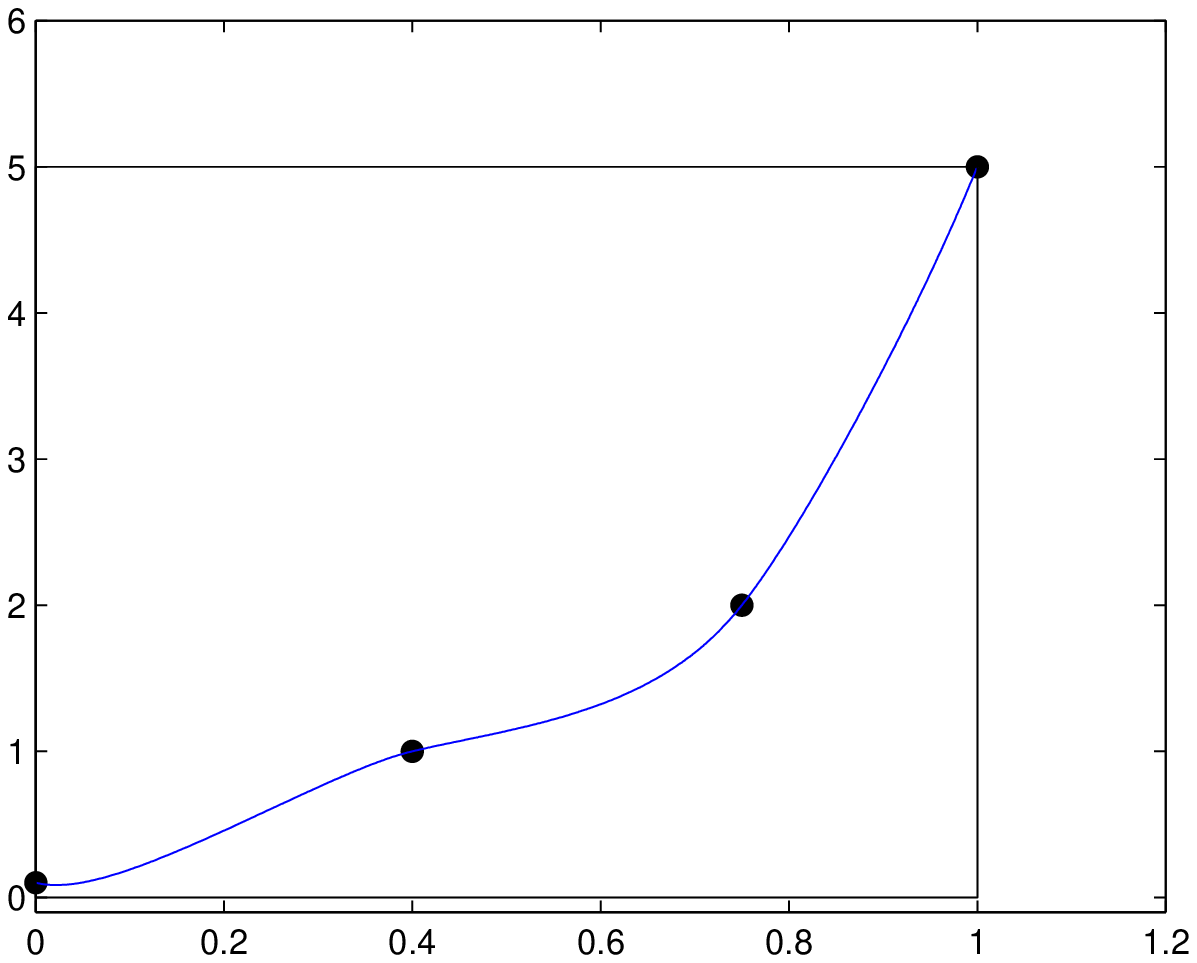,scale=0.29} \centering{\\(e): Classical Rational Cubic Spline within $[0,1]\times[0.1,5]$}
\end{minipage}\hfill
\begin{minipage}{0.32\textwidth}
\epsfig{file =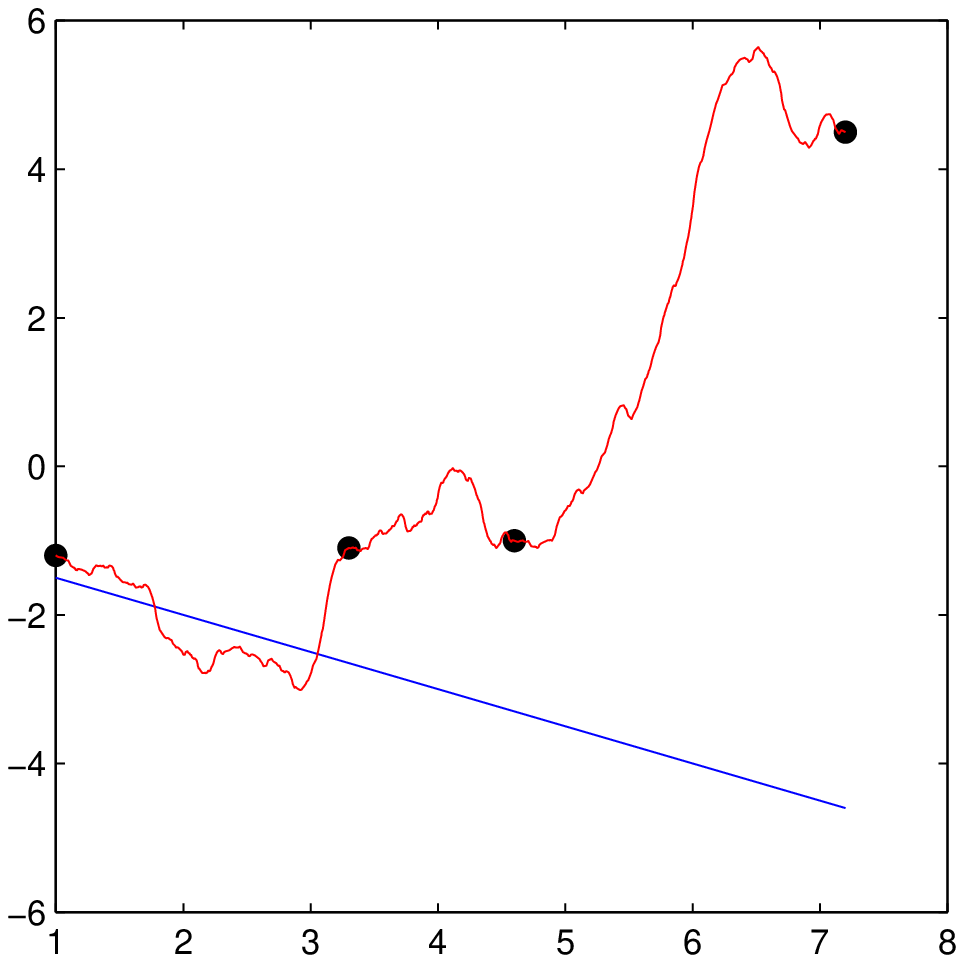,scale=0.29} \centering{\\(f): Rational FIF and the line $y=-0.5x-1$}
\end{minipage}\hfill
\begin{minipage}{0.32\textwidth}
\epsfig{file=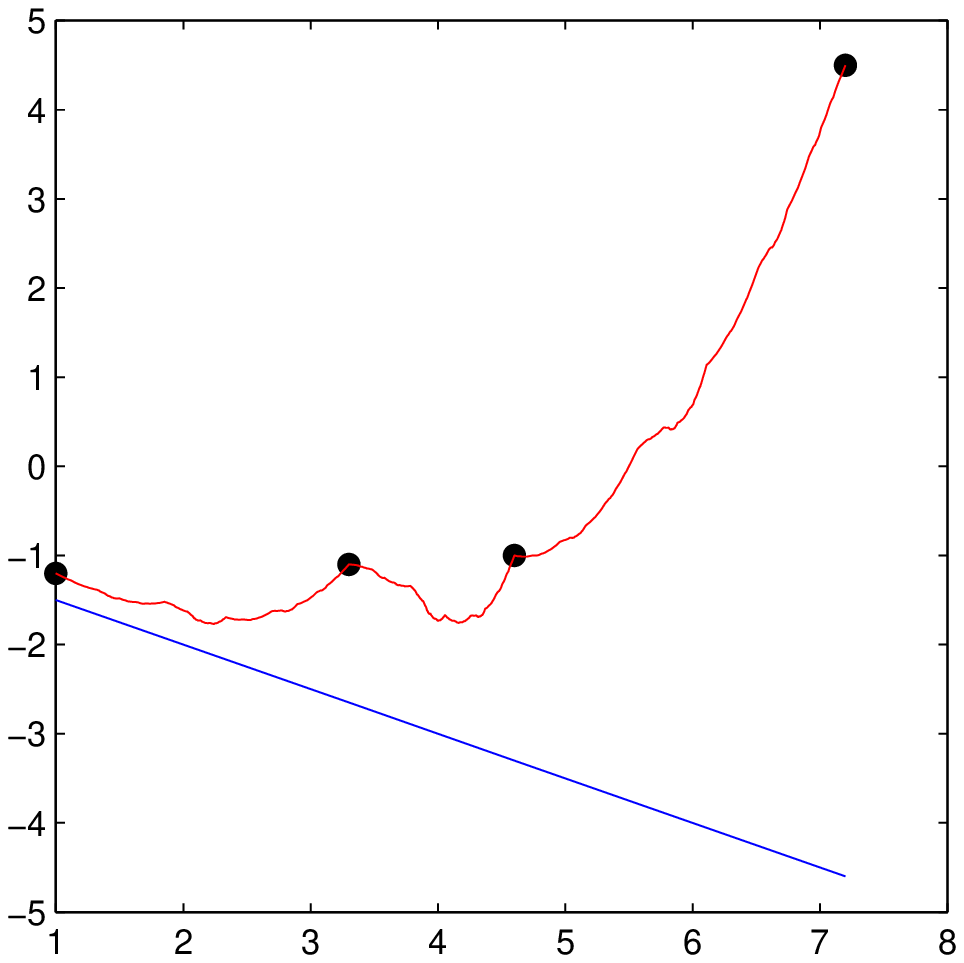,scale=0.29} \centering{\\(g): Rational FIF above the line $y=-0.5x-1$}
\end{minipage}\hfill
\begin{minipage}{0.32\textwidth}
\epsfig{file=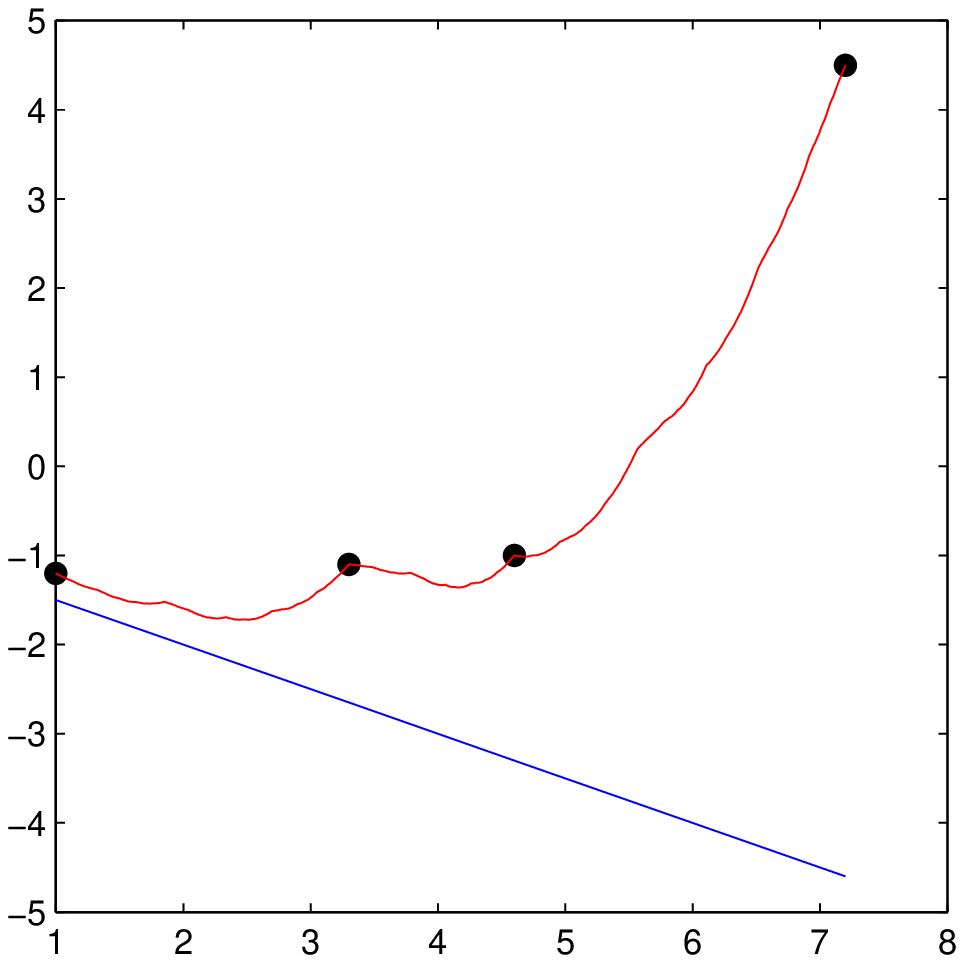,scale=0.29} \centering{\\(h): Rational FIF above the line $y=-0.5x-1$}
\end{minipage}\hfill
\begin{minipage}{0.32\textwidth}
\epsfig{file=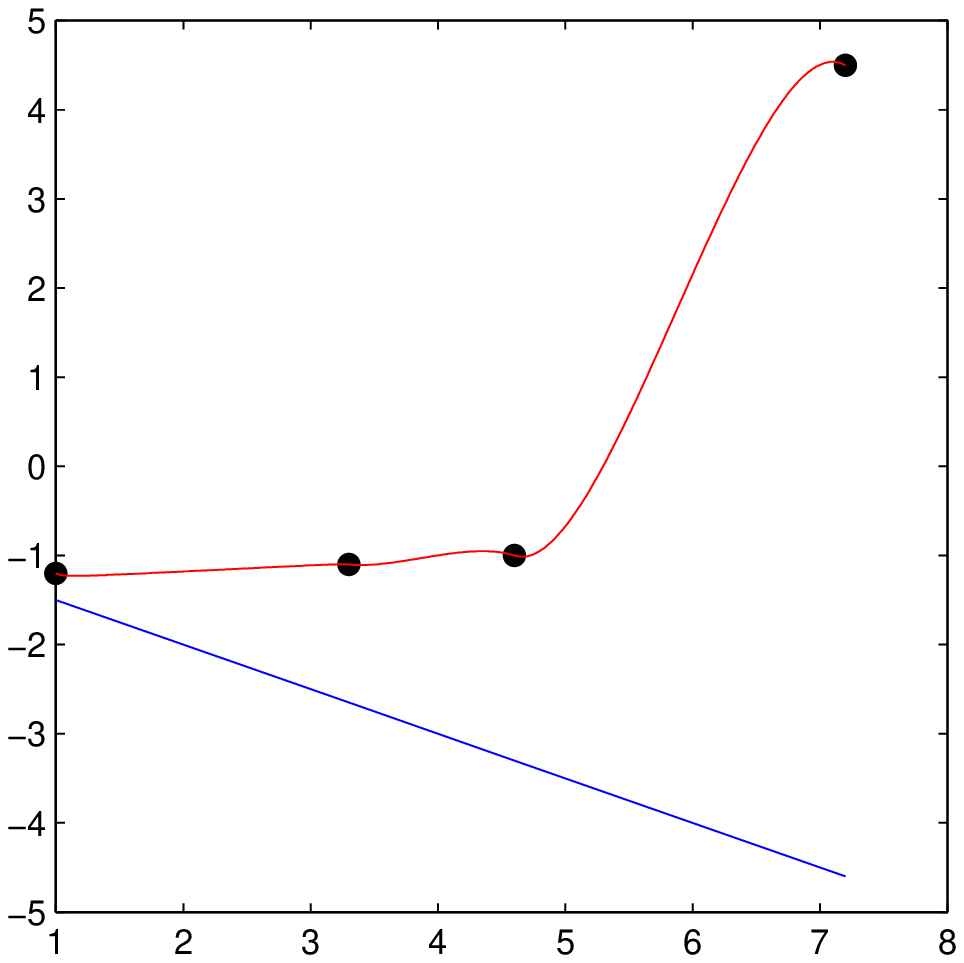,scale=0.29} \centering{\\(i): Classical Rational Cubic Spline above the line $y=-0.5x-1$}
\end{minipage}\hfill\\
\caption{Different types of constraint RCSFIFs) (the
interpolating data points are given by the circles.)}\label{r2ppfig2}
\end{center}
\end{figure}
\section{Conclusion}\label{HALDIASANGEETAsec7}
In this paper, we have constructed RCSFIF with two family of shape parameters. We identify scaling factors and shape parameters to obtain RCSFIF lies inside a rectangle as well as above a prescribed straight line. The scaling parameters and shape parameters play an important role in determining the shape of a RCSFIF. Thus, according to the need of an experiment for simulating objects with smooth geometrical shapes, a large flexibility in the choice of a suitable interpolating smooth fractal interpolant is offered by our approach. As in the case of vast applications of classical rational interpolants in CAM, CAD, and other mathematical, engineering applications, it is felt that RCSFIFs can find rich applications in some of these areas. Further, as classical piecewise cubic Hermite interpolant, $\mathcal{C}^1$-cubic Hermite FIF \cite{CV2}, and $\mathcal{C}^1$-rational cubic spline \cite{SHH} are special cases of RCSFIFs, it is possible to use RCSFIFs for mathematical and engineering problems where these 
approaches does not work satisfactorily. The upper bound for the error between the original function $\Phi$ and the RCCHFIF $\Psi$ is deduced. The authors are in the process of completing to extend this work for the 3D data which will appear somewhere else.

\end{document}